\documentclass[12pt,a4paper,oneside]{amsart}

\usepackage{amsfonts, amsmath, amssymb, amsthm, amscd}
\usepackage{hyperref}
\hypersetup{
  colorlinks   = true, 
  urlcolor     = blue, 
  linkcolor    = blue, 
  citecolor   = red 
}
\usepackage{anysize}
\marginsize{1.8cm}{1.8cm}{1.8cm}{1.8cm}
\usepackage[pdftex]{graphicx}

\newtheorem{theorem}{Theorem}
\newtheorem{lemma}[theorem]{Lemma}

\theoremstyle{definition}

\theoremstyle{remark}
\newtheorem{remark}[theorem]{Remark}

\numberwithin{equation}{section}

\renewcommand{\epsilon}{\varepsilon}

\renewcommand{\phi}{\varphi}

\newcounter{fig}
\newcommand{\f}{\refstepcounter{fig} Fig. \arabic{fig}. }

\title[Any cyclic quadrilateral can be inscribed in any closed convex smooth curve
]{Any cyclic quadrilateral can be inscribed \\in any closed convex smooth curve
}

\author{Arseniy Akopyan} 

\author{Sergey Avvakumov} 

\address{Arseniy Akopyan} \email{akopjan@gmail.com}
\address{Sergey Avvakumov}  \email{s.avvakumov@gmail.com}
\address{Institute of Science and Technology Austria (IST Austria), Am Campus 1, 3400 Klosterneuburg, Austria}

\begin{document}

\begin{abstract}
We prove that any cyclic quadrilateral can be inscribed in any closed convex \mbox{$C^1$-curve}.
The smoothness condition is not required if the quadrilateral is a rectangle.
\end{abstract}

\maketitle

\section{Introduction}

To \emph{inscribe} a quadrilateral $Q$ in a Jordan curve is to find a non-degenerate scaling, rotation, and translation which maps all of the vertices of $Q$ to the curve.

Which quadrilaterals can be inscribed in any closed convex curve? Obviously, the quadrilateral must be \emph{cyclic}, that is inscribed in a circle. Such quadrilaterals are characterized by a remarkable property that the sum of their opposite angles is $\pi$. It turns out that for $C^1$-curves this condition is also sufficient.\footnote{The results of this paper were independently obtained by B.~Matschke, whose preprint~\cite{matschke2018quadrilaterals} appeared just a few days later. In addition it contains a stronger version of Theorem~\ref{thm:rectangle theorem}, which is proved not only for rectangles but also for any cyclic trapezoids.}

\begin{theorem}
	\label{thm:mainthm}
Any cyclic quadrilateral can be inscribed in any closed convex $C^1$-curve.
\end{theorem}

The $C^1$-smoothness requirement is necessary in Theorem~\ref{thm:mainthm}. 
For example, the kite with angles $\pi/2$ and $2\pi/3$ cannot be inscribed in the thin triangle with angles $\pi/10$, $\pi/10$, $4\pi/5$ (Figure~\ref{fig:nonsmoothness counterexample}). 
However, if $Q$ is a rectangle the smoothness condition can be relaxed.
	
\begin{theorem}
	\label{thm:rectangle theorem}
Any rectangle can be inscribed in any closed convex curve.
\end{theorem}

	\begin{center}
	\includegraphics{fig-quadrileteral-9.mps}\\
	\f \label{fig:nonsmoothness counterexample} 
	\end{center}

In \cite{makeev1995quadrangles} V. Makeev conjectured that any cyclic quadrilateral can be inscribed in any Jordan curve. He proved the conjecture for the case of star-shaped $C^2$-curves intersecting any circle at no more than $4$ points. Theorem~\ref{thm:mainthm} proves it for convex $C^1$-curves. The example above shows that the conjecture fails without the smoothness assumption. This example can be generalized to any cyclic quadrilateral except for trapezoids. So, I. Pak \cite{pak2008discrete} conjectured that Makeev's conjecture still holds for cyclic trapezoids even without the smoothness assumption.

Makeev's conjecture is a part of a substantial topic originating from the famous question in geometry known as the Square Peg problem or Toeplitz' conjecture: Does every Jordan curve contain all the vertices of a square?  In its general form the Square Peg problem is still open. In the last hundred years it was positively solved, however, for a wide  variety of classes of curves. For instance, for convex curves and later for piecewise analytic curves by A.~Emch~\cite{emch1913some,emch1916onsome}, for $C^2$-curves by L.~Schnirelman \cite{shnirelman1944certain}, for locally monotone curves by W.~Stromquist~\cite{stromquist1989inscribed}, for curves without special trapezoids, for curves inscribed in a certain annulus, and for centrally symmetric curves \cite{nielsen1995rectangles}. There were also high-dimension extensions of these results~\cite{gromov1969simplexes, hausel2000inscribing, makeev2003universally, karasev2009inscribing}.
For more details we refer the reader to the survey \cite{matschke2014asurvey} by B.~Matschke.

Similar to the Square Peg problem, there exists the Rectangular Peg conjecture stating that every Jordan curve contains the vertices of a rectangle with the prescribed aspect ratio. A proof was claimed by  H. Griffiths~\cite{griffiths1991thetopology}, but an error was found later. For a discussion see \cite[Conjecture~8]{matschke2014asurvey}.
The specific case of aspect ratio $\sqrt{3}$ and ``close to convex'' curves was solved B.~Matschke~\cite{matschke2011equivariant}. Theorem~\ref{thm:rectangle theorem} proves the Rectangular Peg conjecture for the case of convex curves.

It is noteworthy that all of the proofs mentioned above use the topological obstruction theory. Unfortunately, this approach fails in the more general cases. The proof of Theorem~\ref{thm:mainthm} and \ref{thm:rectangle theorem} is based on a ``non-topological'' observation first made by R. Karasev, \cite{karasev2016anote}. It allowed him, in particular, to prove the infinitesimal version of Theorem~\ref{thm:mainthm}. R. Karasev noticed and proved that during the rotation of any three out of four vertices of a quadrilateral $Q$ along a curve $\gamma$ the fourth vertex travels along a path bounding the same signed area as the original curve $\gamma$. A similar idea was recently independently discovered by T.~Tao who used it to prove the Toeplitz' conjecture for new types of curves \cite{tao2017integration}.

\section{The case of strictly convex $C^\infty$-curves.}

In this section we prove the following theorem.

\begin{theorem}
	\label{thm:mainthm'}
Let $Q$ be a cyclic quadrilateral and let $\gamma$ be a closed strictly convex $C^\infty$-curve. Then for any $\varepsilon > 0$ there is a cyclic quadrilateral $Q_\varepsilon$ which is $\varepsilon$-close to $Q$ and can be inscribed in $\gamma$.
\end{theorem}

The proofs of Theorems~\ref{thm:mainthm} and \ref{thm:rectangle theorem} are obtained from Theorem~\ref{thm:mainthm'} by ``going to the limit'' type argument in the next section.

Until the end of the section let us fix a closed strictly convex $C^\infty$-curve $\gamma$.

For a quadrilateral $Q$ its \emph{drawing} is the image of $Q$ under some non-degenerate scaling, rotation, and translation. If the angle of the rotation is $\alpha$ we also call it an \emph{$\alpha$-drawing}.

\begin{lemma}
	\label{lem:inscribed triangle}
Pick a vertex of a quadrilateral $Q$. For any angle $\alpha$ there is a unique $\alpha$-drawing of~$Q$ with all of the remaining $3$ vertices being on $\gamma$.
\end{lemma}
\begin{proof}
The existence of a drawing follows by a simple continuity argument similar to the argument in the proof of the following Lemma \ref{lem:diagonal}. The drawing is unique because the vertices of two distinct homothetic triangles cannot lie on a {\emph strictly} convex curve.
\end{proof}

For a vertex $d$ of a quadrilateral $Q$ denote by $d(\alpha)$ the position of $d$ in the $\alpha$-drawing of $Q$ with the remaining $3$ vertices being on $\gamma$.

\begin{lemma}
\label{lem:sard}
Let $Q$ be a cyclic quadrilateral. Then for any $\varepsilon > 0$ there is a cyclic quadrilateral~$Q_{\varepsilon}$ which is $\varepsilon$-close to $Q$ and such that $d(\alpha)$ is a closed $C^\infty$-curve for any vertex $d$ of $Q_{\varepsilon}$.
\end{lemma}
\begin{proof}
Let $U$ be the space of ordered triples of pairwise distinct points of $\gamma$. Consider the map $f:U\rightarrow S^1\times S^1$ which sends a triple $(x,y,z)$ to the pair of angles $(\angle xyz, \angle yzx)$. Clearly, $f$ is~$C^\infty$.

Let $abcd$ be the vertices of $Q$ in the counterclockwise order. Then the curve $d(\alpha)$ corresponds to the $f$-preimage of the pair of angles $(\angle abc,\angle bca)$. By Sard's lemma the set of the critical values of $f$ has Lebesgue measure $0$, i.e., the set of cyclic quadrilaterals $Q$ such that $d(\alpha)$ is \emph{not} a $C^\infty$-curve also has Lebesgue measure $0$. Applying this argument to every vertex of $Q$ we get the statement of the lemma.
\end{proof}

\begin{lemma}
\label{lem:diagonal}
Let $Q$ be a cyclic quadrilateral. Then there is a vertex $d$ of $Q$ such that the angle at~$d$ is non-acute and $d(\alpha)$ contains a point either on $\gamma$ or in the exterior of $\gamma$.
\end{lemma}

\begin{proof}
Let $abcd$ be the vertices of $Q$ in the counterclockwise order. 

There are two adjacent vertices of $Q$ with non-acute angles because $Q$ is cyclic. Without the loss of generality let us assume that the angles at $c$ and $d$ are non-acute. We may also assume that $\gamma$ is tangent to the lines $y=0$ and $y=1$.

For $t\in (0,1)$ denote by $a_t$ and $b_t$ the leftmost and the rightmost, respectively, of the two intersections of $\gamma$ with $y=t$.
Denote by $c_t$ and $d_t$ the points such that $a_tb_tc_td_t$ is a drawing of~$Q$.
Note that $c_t$ and $d_t$ are above the line $y=t$, see Figure~\ref{fig:inout}.

Consider the case when $t$ is very close to $1$. Then $\angle d_ta_tb_t$ is greater than the angle between $a_tb_t$ and the tangent to $\gamma$ at $a_t$. Which places $d_t$ is in the exterior of $\gamma$. Likewise, $c_t$ is also in the exterior of $\gamma$.

Consider now the opposite case of $t$ being very close to $0$. Then $\angle d_ta_tb_t$ is less than the angle between $a_tb_t$ and the tangent to $\gamma$ at $a_t$. Also, the segment $a_td_t$ is ``short'', i.e., much shorter than the intersection of the interior of $\gamma$ with the line parallel to $a_td_t$ and going through the common point of $\gamma$ and $y=0$. Which means that $d_t$ is in the interior of $\gamma$. Likewise, $c_t$ is also in the interior of $\gamma$.

Let us now continuously decrease $t$ from $1$ to $0$. At some moment one of the vertices $d$ or $c$ is going to intersect $\gamma$ while another one is still in the exterior of $\gamma$, or on $\gamma$. Without the loss of generality we may assume that the latter vertex is $d$. Then $d(\alpha)$ contains a point either on $\gamma$ or in the exterior of $\gamma$.
\end{proof}

	\parbox[b]{0.4\textwidth}{
	\begin{center}
	\includegraphics{fig-quadrileteral-12.mps}\\
	\f \label{fig:inout} 
	\end{center}
	}
	\hskip 0.3cm
	\parbox[b]{0.5\textwidth}{
	\begin{center}
	\includegraphics{fig-quadrileteral-10.mps}\\
	\f \label{fig:for ellipse} 
	\end{center}
	}

\begin{proof}[Proof of Theorem \ref{thm:mainthm'}.]
Choose a cyclic quadrilateral $Q_{\varepsilon}$ as in the statement of Lemma \ref{lem:sard}.

By Lemma \ref{lem:diagonal}, there is a vertex $d$ of $Q_{\varepsilon}$ with a non-acute angle and such that $d(\alpha)$ contains a point either on $\gamma$ or in the exterior of $\gamma$. 

If $d(\alpha)$ intersects $\gamma$ then we are done, so we may assume that $d(\alpha)$ and $\gamma$ are disjoint. 

By the Jordan curve theorem, $d(\alpha)$ lies in the exterior of $\gamma$. On the other hand, $\gamma$ lies in the interior of $d(\alpha)$. Indeed, as we revolve once along $d(\alpha)$, the diagonal $bd$ of the corresponding $\alpha$-drawing of $Q_{\varepsilon}$ must also complete a $2\pi$ rotation. This would be impossible if the interiors of~$\gamma$ and $d(\alpha)$ were disjoint.

By the following lemma, the curve $d(\alpha)$ has no self-intersections in the exterior of $\gamma$, i.e., no self-intersections at all.
\begin{lemma}
	\label{lem:similar quadrileterals}
	Let $abcd$ and $a'b'c'd$ be two distinct drawings of the same cyclic quadrilateral with a non-acute $\angle d$.
	Suppose that $d$ is outside of the convex hull of the points $a$,$b$,$c$,$a'$,$b'$,and $c'$. 
	Then six points $a$, $b$, $c$ and $a'$, $b'$, $c'$ cannot be in strictly convex position.
\end{lemma}

It is left to note that the area of $d(\alpha)$ must then be greater than the area of $\gamma$ which contradicts to the following lemma proven in \cite{karasev2016anote}.
\begin{lemma}
	\label{lem:area}
Let $Q$ be a cyclic quadrilateral and let $d(\alpha)$ be a $C^\infty$-curve for some vertex $d$ of $Q$. Then the area of $d(\alpha)$ is equal to the area of $\gamma$.
\end{lemma}

\end{proof}

\section{Proofs of Theorem~\ref{thm:mainthm} and Theorem~\ref{thm:rectangle theorem}.}

\begin{proof}[Proof of Theorem~\ref{thm:mainthm}]
Let $\gamma$ be a closed convex $C^1$-curve and let $Q$ be a cyclic quadrilateral. Choose a sequence $\gamma_i$ of closed strictly convex $C^\infty$-curves converging to $\gamma$ pointwise. By Theorem~\ref{thm:mainthm'}, there is a sequence $Q_i$ of cyclic quadrilaterals converging to $Q$ and such that $Q_i$ can be inscribed in $\gamma_i$ for each $i$.

Let $a_i,b_i,c_i,d_i\in\gamma_i$ be the vertices of $Q_i$ in the counterclockwise order. By passing to the subsequences, we may assume that the sequences $a_i$, $b_i$, $c_i$, and $d_i$ have limits, $a$, $b$, $c$, and $d$, respectively.

The quadrilateral $abcd$ is inscribed in $\gamma$ and is obtained from $Q$ by a composition of scaling, rotation, and translation. It remains to prove that the scaling is non-degenerate.

Assume to the contrary that $a=b=c=d$. Then both $a_ib_i$ and $b_ic_i$ converge to the tangent to $\gamma$ at $a=b=c=d$. I.e., $\angle a_ib_ic_i$ converges to either $0$ or $\pi$. On the other hand, $\angle a_ib_ic_i$ must converge to the corresponding angle of $Q$. Angles of $Q$ are neither $0$ nor $\pi$, which leads to a contradiction.
\end{proof}

\begin{proof}[Proof of Theorem~\ref{thm:rectangle theorem}]
Let $\gamma$ be a closed convex curve and let $Q$ be a rectangle. Let $\gamma_i$, $Q_i$, $a_i$, $b_i$, $c_i$, $d_i$, $a$, $b$, $c$, and $d$ be as in the proof of Theorem~\ref{thm:mainthm}.

Again, it remains to prove that the rectangle $abcd$ is non-degenerate.
Assume to the contrary, that $a=b=c=d$.

The vertices of $Q_i$ divide $\gamma_i$ into $4$ arcs which we denote in the counterclockwise order by $A_i$, $B_i$, $C_i$, and $D_i$, see Figure~\ref{fig:rect}. The lengths of $3$ out of $4$ arcs converge to $0$. Without the loss of generality we assume that the lengths $l(A_i)$ and $l(C_i)$ of $A_i$ and $C_i$, respectively, converge to $0$.

Denote by $L_{A_i}$ and $L_{C_i}$ the tangents to $A_i$ and $C_i$ parallel to $a_ib_i$, see Figure~\ref{fig:rect}. The distance between $L_{A_i}$ and $L_{C_i}$ is less than $l(A_i) + l(C_i) + |b_ic_i|$, i.e., it converges to $0$. Which contradicts to the fact that the curve $\gamma_i$ lies between $L_{A_i}$ and $L_{C_i}$ for each $i$.
\end{proof}

	\begin{center}
	\includegraphics{fig-quadrileteral-11.mps}\\
	\f \label{fig:rect} 
	\end{center}

\section{Proof of Lemma~\ref{lem:similar quadrileterals}}

Assume to the contrary of the statement of the lemma that the points $a$, $b$, $c$, $a'$, $b'$, and $c'$ lie in strictly convex position.

	Draw the lines containing the sides of the triangle $abc$, and denote the angular regions of the plane formed by them as shown in Figure \ref{fig:halfplane}.
	At first, let us note that the point $a'$ cannot belong to the region $C_a$, because in that case $a$ is covered by the triangle $a'bc$.
	Analogously, the points $b'$ and $c'$ cannot belong to the regions $C_b$ and $C_c$, respectively.

	Denote by $\Omega$ the circumcircle of $abcd$ and let $\ell$ be the tangent line to $\Omega$ at $a$.
	Together with the lines $ab$ and $ac$ the line $\ell$ forms two additional angular regions denoted by $C_b'$ and $C_c'$, respectively, see Figure~\ref{fig:halfplane}. 
	
	It is easy to see that the composition of a homothety and a rotation around $d$ which sends $a$ to $b$ and $a'$ to $b'$ also sends $C_b'$ to $C_b$. So, if $a'$ lies in $C_b'$ then $b'$ lies in $C_b$, which we already showed to be impossible. Therefore $a'$ cannot lie in $C_b'$. Analogously, $a'$ cannot lie in $C_c'$.
	
	We proved that $a'$ cannot lie in the union of $C_a$, $C_b'$, and $C_c'$, which means that $a'$ and $d$ lie on the same side of the line $ab$. Similarly, $a$ and $d$ lie on the same side of the line $a'b'$. From the latter fact we	 conclude that $a'$ lies in the exterior of the circle $\omega$ passing through $d$ and tangent to $ab$ at $a$ (Figure~\ref{fig:circle}). To see this the reader might consider that $a'b'$ passes through $a$ iff $a'$ belongs to $\omega$.

	We know that $d$ lies outside of the convex hull of $a$, $b$, $c$, $a'$, $b'$, and $c'$ and that $\angle cda=\angle c'da'\geq \pi/2$,  so $\angle ada' < \pi/2$. Thus, without the loss of generality we may assume that the counterclockwise rotation sending $da$ to $da'$ is at most $\pi/2$. Therefore, $a'$ lies in the angular region formed by $\angle cda$.
	
	This region is covered by the four grey zones in Figures~\ref{fig:halfplane},~\ref{fig:circle},~\ref{fig:segment rotation inside},~\ref{fig:segment rotation outside}. Let us give a verbal description of the zones and prove that $a'$ cannot lie in them. This will conclude the proof of the lemma. 

	\parbox[b]{0.4\textwidth}{
		\begin{center}
			\includegraphics{fig-quadrileteral-6.mps}\\
		\f \label{fig:halfplane}
		\end{center}
	}
	\hskip 0.3cm
	\parbox[b]{0.5\textwidth}{
		\begin{center}
			\includegraphics{fig-quadrileteral-7.mps}\\
		\f \label{fig:circle}
		\end{center}
	}
	
	\begin{itemize}
	\item The grey zone in Figure~\ref{fig:halfplane} is the halfplane bounded by $ab$ and not containing $d$.
	\item The grey zone in Figure~\ref{fig:circle} is the interior of $\omega$.
	\end{itemize}
	We have already proved that $a'$ cannot lie in the grey zones in Figures~\ref{fig:halfplane} and~\ref{fig:circle}.
	
	\begin{itemize}
	\item The grey zone in Figure~\ref{fig:segment rotation inside} is the intersection of the halfplane bounded by $ab$ and containing $d$, the exterior of $\omega$, and the interior of $\Omega$. Suppose that $a'$ is in that grey zone.
	
	Denote by $x$ the point of intersection of the line $ab$ with the circle going through $a$, $a'$, and $d$. Point $x$ is inside of the segment $ab$. From the angular property of an iscribed quadrilateral it follows that $b'$, $a'$, and $x$ lie on the same line. So, the intersection point $x$ of $ab$ and $a'b'$ lie inside of the segment $ab$ and outside of the segment $a'b'$ which contradicts to the convex position of $a$, $b$, $a'$, and $b'$.
	
	\item The grey zone in Figure~\ref{fig:segment rotation outside} is the intersection of the halfplane bounded by $\ell$ and containing $c$, the halfplane bounded by $ac$ and not containing $d$, and the exterior of $\Omega$. Suppose that $a'$ is in that grey zone.
	
	Denote by $y$ the point of intersection of the line $ac$ with the circle going through $a$, $a'$, and $d$. Point $y$ is outside of the segment $ac$. From the angular property of an iscribed quadrilateral it follows that $c'$, $a'$, and $y$ lie on the same line. So, the intersection point $y$ of $ac$ and $a'c'$ lie outside of the segment $ac$ and inside of the segment $a'c'$ which contradicts to the convex position of $a$, $c$, $a'$, and $c'$.
	\end{itemize}

\begin{remark}
	\label{rem: counterexamples}
	 Lemma~\ref{lem:similar quadrileterals} does not hold if the quadrilateral is not assumed to be cyclic (Figure~\ref{fig:non-inscribed-convex}) or if the angle at $d$ is allowed to be acute (Figure~\ref{fig:acute-convex}).
\end{remark}
\subsection*{Acknowledgments}
The first author is supported by the European Research Council (ERC) under the European Union's Horizon 2020 research and innovation programme (grant agreement  No 716117). 

The authors are grateful to Roman Karasev for useful discussions and ideas.

	\parbox[b]{0.4\textwidth}{
		\begin{center}
			\includegraphics{fig-quadrileteral-4.mps}\\
		\f \label{fig:segment rotation inside} 
		\end{center}
	}
	\hskip 0.3cm
	\parbox[b]{0.5\textwidth}{
		\begin{center}
			\includegraphics{fig-quadrileteral-5.mps}\\
		\f \label{fig:segment rotation outside} 
		\end{center}
	}

	\parbox[b]{0.4\textwidth}{
		\begin{center}
			\includegraphics{fig-quadrileteral-1.mps}\\
		\f \label{fig:non-inscribed-convex} 
		\end{center}
	}
	\hskip 0.3cm
	\parbox[b]{0.5\textwidth}{
		\begin{center}
			\includegraphics{fig-quadrileteral-2.mps}\\
		\f \label{fig:acute-convex}	
		\end{center}
	}

%

\bibliography{quadri}
\bibliographystyle{abbrv}

\end{document}